\documentclass[final]{siamltex}

\usepackage{color}
\usepackage{CJK}

\usepackage{amsmath, amsfonts, amssymb,latexsym, mathrsfs, tikz, booktabs}

\usepackage{graphicx}
\usepackage{geometry}
\usepackage{hyperref}
\usepackage{color}
\usepackage{bm}
\usepackage{multicol}
\usepackage{multirow}

\usepackage[title]{appendix}

\newcommand{\sfrac}[2]{\mathchoice
  {\kern0em\raise.5ex\hbox{\the\scriptfont0 #1}\kern-.15em/
   \kern-.15em\lower.25ex\hbox{\the\scriptfont0 #2}}
  {\kern0em\raise.5ex\hbox{\the\scriptfont0 #1}\kern-.15em/
   \kern-.15em\lower.25ex\hbox{\the\scriptfont0 #2}}
  {\kern0em\raise.5ex\hbox{\the\scriptscriptfont0 #1}\kern-.2em/
   \kern-.15em\lower.25ex\hbox{\the\scriptscriptfont0 #2}}
  {#1\!/#2}}

\def\gb {\bm{g}}

\def\ub {\bm{u}}
\def\vb {\bm{v}}

\def\hb {\bm{h}}
\def\nb {\bm{n}}

\def\div{\mbox{div}}

\def\det{\mbox{det}}

\def\epsilonb {\boldsymbol{\epsilon}}

\begin{document}

\title{Schur complement based preconditioners for twofold and block tridiagonal saddle point problems}
\author{Mingchao Cai \thanks{Corresponding author, Department of Mathematics, Morgan State University, Baltimore, MD 21251, USA, E-mail: cmchao2005@gmail.com}
\and
Guoliang Ju \thanks{Solid State R \& D Department, Shenzhen Tenfong CO. LTD., Shenzhen, 518055, China,
E-mail: jugl@tenfong.cn}
\and
Jingzhi LI \thanks{Department of Mathematics, Southern University of Science and Technology, Shenzhen, Guangdong 518055, China, E-mail: li.jz@sustech.edu.cn}
}

\date{}          
\maketitle

\begin{abstract}
In this paper, we consider using Schur complements to design preconditioners for twofold and block tridiagonal saddle point problems. One type of the preconditioners are based on the
nested (or recursive) Schur complement, the other is based on an additive type Schur complement after permuting the original saddle point systems. We analyze different preconditioners incorporating the exact Schur complements. We show that some of them will lead to positively stable preconditioned systems if proper signs are selected in front of the Schur complements.
These positive-stable preconditioners outperform other preconditioners if the Schur complements are further approximated inexactly. Numerical experiments for a 3-field formulation of the Biot model are provided to verify our predictions.
\end{abstract}

{\bf Keywords} Schur complement; block tridiagonal systems; positively stable preconditioners; Routh–Hurwitz stability criterion.

\section{Introduction}
Many application problems will lead to twofold and/or block tridiagonal saddle point linear systems. Important examples include mixed formulations of the Biot model \cite{cai2022overlapping, hong2020parameter, lee2017parameter, oyarzua2016locking}, the coupling of fluid flow with porous media flow \cite{cai2009preconditioning, gatica2012twofold}, hybrid discontinuous Galerkin approximation of Stokes problem \cite{cockburn2014devising}, liquid crystal problem \cite{beik2018block, Ramage2013preconditioner} and optimization problems
\cite{howell2011infsup, lee2017parameter, mardal2017robust, Pearson2012new, Sogn2019Schur}. Some of these problems (or after permutations) will lead to a twofold saddle point problem \cite{beik2018block, cao2019shift, gatica2012twofold, howell2011infsup, huang2019spectral, Sogn2019Schur, wang2021augmented, xie2020note} (or the so-called double saddle point problem) of the following form.
\begin{equation}\label{matrix_A}
\mathcal{A}{\bm x}=
\left[\begin{array}{ccc}
A_{1}     & B^T_{1} & 0 \\[1mm]
C_1 &-A_{2}     & B^T_{2} \\[1mm]
 0             &C_{2}     &A_3 \\
\end{array}\right]
\left[\begin{array}{c}
{\bm x}_1\\[1mm]
{\bm x}_2\\[1mm]
{\bm x}_3
\end{array}\right]=
\left[\begin{array}{c}
 {\bm f}_{1}\\[1mm]
{\bm f}_{2}\\[1mm]
{\bm f}_{3}
\end{array}\right].
\end{equation} 
A negative sign in front of $A_{2}$ is just for the ease of notation.
After simple permutations, the system matrix of (\ref{matrix_A}) can be rewritten into the following form.
\begin{equation}\label{permuted_form1}
\mathcal{A}=
\left[\begin{array}{ccc}
A_{1}     & 0& B^T_{1} \\[1mm]
0   &A_3     & C_{2} \\[1mm]
C_{1}            &B_{2}^T     &-A_{2}\\
\end{array}\right].
\end{equation}
We call the system matrix in (\ref{permuted_form1}) permutation-equivalent to that in (\ref{matrix_A}). Without causing confusion, we continue to
use the notation $\mathcal{A}$ for the permuted matrix (\ref{permuted_form1}).
The linear system in (\ref{permuted_form1}) arises naturally from the domain decomposition methods \cite{mathew2008domain, Toselli2006domain}.
In this work, we only assume that $A_1$ is invertible and the global system matrix $\mathcal{A}$ is invertible. Many special cases can be cast into the
above forms of twofold saddle point systems. For example, (a) $A_2=0$; (b) $A_3=0$; (c) both $A_2=0$ and $A_3=0$. Our discussions will try to cover all these special cases.

The above $3$-by-$3$ block linear problems (\ref{matrix_A}) and  (\ref{permuted_form1}) can be naturally extended to the $n$-tuple cases. For example, when the system matrix in (\ref{matrix_A}) is extended to the $n$-tuple case, it is the block tridiagonal systems discussed in \cite{Sogn2019Schur}. When
the system matrix in (\ref{permuted_form1}) is extended to the $n$-tuple case, it corresponds to the linear system resulting from the domain decomposition method for elliptic
problems with $n-1$ subdomains.  In many references, these linear systems are assumed to be symmetric. No matter whether it is symmetric or not, $\mathcal{A}$ is generally indefinite. For solving such a system in large-scale computations, Krylov subspace methods with preconditioners are usually applied. The analysis in \cite{murphy2000note, ipsen2001note} indicates that one should employ Schur complement based preconditioners \cite{benzi2005numerical, cao2006class, cao2007positive, lin2007note} and then resort to replacing each exact Schur complement with an approximate one. If a Krylov subspace method satisfies an optimality or Galerkin property (such as conjugate gradient method and minimal residual methods \cite{Eiermann2001geometric, murphy2000note, Saad2003iterative}), the number of iterations depends on the degree of the minimal polynomial that the preconditioned matrix $\mathcal{T}$ satisfies \cite{murphy2000note, ipsen2001note, Saad2003iterative}. More precisely, the number of iterations depends on the dimension of the Krylov subspace $\mbox{Span}\{{\bm r},\mathcal{T}{\bm r},\mathcal{T}^2{\bm r},\mathcal{T}^3{\bm r},\ldots\}$. For a classical $2$-by-$2$ saddle point problem, by employing block-diagonal or block-triangular preconditioners with the $(2, 2)$ block being the Schur complement, it has been shown that the corresponding preconditioned system satisfies a polynomial of the degree of $4$ or $2$ \cite{murphy2000note, ipsen2001note}. Therefore, devising preconditioners based on Schur complements is of vital importance.

In this study, we explore two methodologies for designing preconditioners tailored for $3$-by-$3$ block systems exhibiting block tridiagonal form, as well as their extensions to $n$-tuple scenarios. One approach centers around the nested (or recursive) Schur complement \cite{Sogn2019Schur}. Unlike the approach presented in \cite{Sogn2019Schur}, our work adopts a more expansive assumption, allowing $A_1$, $A_2$, and $A_3$ to be non-symmetric or negative definite. 
Another approach is founded on the permuted matrix (\ref{permuted_form1}). Initially treating the first 2-by-2 block as a unified entity, albeit potentially indefinite, we transform the system into the saddle point form discussed in \cite{ipsen2001note}. Subsequently, we explore preconditioners based on an additive-type Schur complement \cite{cao2022additive, mathew2008domain, Toselli2006domain}. For $3$-by-$3$ block systems, the additive-type Schur complement-based preconditioner yields a system that satisfies a polynomial of degree 2 (block triangular) or 4 (block diagonal), while preconditioners based on the nested Schur complement satisfy a polynomial with a degree of 3 (block triangular) or 6 (block diagonal).
When extended to $n$-tuple scenarios, if the preconditioned system is diagonalizable, the degree of the polynomial satisfied by a nested Schur complement-based preconditioner is significantly higher than that satisfied by an additive Schur complement-based preconditioner. Thus, from an algebraic perspective, when employing the GMRES method, the additive-type preconditioner is favored. The generalization of the form (\ref{permuted_form1}) to $n$-tuple cases corresponds to the linear system resulting from domain decomposition methods, referencing some literature results on domain decomposition methods.

Commencing with a twofold saddle point problem, we generalize our theory to $n$-tuple block tridiagonal saddle point problems. Our study demonstrates that judiciously selecting signs in front of Schur complements in preconditioners results in a positively stable preconditioned system \cite{cao2007positive}. By using the Routh–Hurwitz stability criterion, we prove that the appropriate sign in front of each Schur complement yields a positively stable preconditioned system \cite{cao2007positive}. For positively stable preconditioners, we highlight that they outperform other preconditioners if further inexact approximations are applied to the corresponding Schur complements.
Numerical experiments, exemplified by a 3-field formulation of the Biot model, underscore the superior performance of positively stable preconditioners compared to other alternatives. Notably, we refrain from assuming the positive definiteness of each $A_i$, and the linear systems investigated in this study extend beyond saddle point problems. In the concluding remarks, we discuss and highlight the potential advantages and disadvantages of the two types of preconditioners, particularly when considering their corresponding inexact versions.

The outline of the remainder of this paper is as follows. In section \ref{sec:NestSchur}, we briefly recall the classic saddle point problem and its Schur complement, and introduce the twofold saddle point problem and the form of Schur complement, we then construct and analyze the block-triangular and block-diagonal preconditioners based on Schur complement for twofold saddle point problems. Furthermore, we extend these results to the $n$-tuple saddle point problem in Section \ref{subsec:NestGSchur}. Some additive Schur complement based preconditioners are constructed and the corresponding known results in the literature are recalled in Section \ref{sec:AddSchur} for twofold saddle point problems. Generalizations to $n$-tuple cases are provided in Section 5. In Section 6, numerical experiments for a 3-field formulation of the Biot model are provided to justify the advantages of using positively stable preconditioners. Finally, concluding remarks are given in Section \ref{sec:Conclusion}.

\section{Nested Schur complement based preconditioners for twofold saddle point systems}\label{sec:NestSchur}
The matrix form of a classic saddle point problem \cite{benzi2005numerical, ipsen2001note} reads as
$$
\mathcal{A}=
\left[\begin{array}{ccc}
A     & B^T\\
C     &-D
\end{array}\right].
$$
Following the notation in \cite{Sogn2019Schur}, the Schur complement of $\mathcal{A}$ is defined as
$$
\mbox{Schur}(\mathcal{A})=-D -CA^{-1}B^T.
$$
In many references, $A$ is assumed to be a symmetric positive definite, if $C$ is assumed to be $B$ and $D$ is zero, then the system is called
KKT system or saddle point system. Usually, $D$ is assumed to be a symmetric and semi-positive definite. In this paper, we only make some assumptions
that can guarantee the invertibility of $\mathcal{A}$. We assume that $A$ and the Schur complement $\mbox{Schur}(\mathcal{A})$
are invertible. For the twofold saddle point problem (\ref{matrix_A}), following the notations in \cite{Sogn2019Schur}, we denote
$$
S_1=A_{1}, \quad S_2=A_2+C_{1}S_{1}^{-1}B_{1}^T,
$$
and denote the nested Schur complement as
$$
S_3=A_3+C_{2}S_2^{-1}B_{2}^T.
$$
In \cite{Sogn2019Schur}, each $S_i$ is assumed to be symmetric positive definite (SPD). From the expressions, if $\mathcal{A}$ is symmetric and all $A_i~i=1,\ldots,3$ are symmetric positive definite, then $S_i$ are SPD. In this work, we only assume that each $S_i$ is invertible. Such assumptions are sufficient enough to guarantee the invertibility of the global system. Under such an assumption, the block $LDU$ decomposition of the system matrix is
\begin{equation}\label{LDUdecomp}
\mathcal{A}=
\left[\begin{array}{ccc}
I     & 0 & 0 \\[1mm]
C_{1}A_{1}^{-1} &I     & 0 \\[1mm]
 0             &-C_{2}S_2^{-1}     &I
\end{array}\right]
\left[\begin{array}{ccc}
A_{1}     & 0 & 0 \\[1mm]
0 &-S_2     & 0 \\[1mm]
 0             &0     &S_3
\end{array}\right]
\left[\begin{array}{ccc}
I     & A_{1}^{-1}B_{1}^T & 0 \\[1mm]
0 &I     & -S^{-1}_2 B_{2}^T\\[1mm]
 0             &0     &I
\end{array}\right].
\end{equation}
For ease of presentation, we will denote the above 3 block matrices as $\mathcal{L}$, $\mathcal{D}$, and $\mathcal{U}$,  respectively.
We study both block-triangular and block-diagonal preconditioners for the system matrix (\ref{matrix_A}). For block-triangular preconditioners, we focus on a lower triangular type with left preconditioning because an upper triangular one with right preconditioning can be discussed in a similar way \cite{benzi2005numerical, ipsen2001note}. We consider the following preconditioner:
\begin{equation}\label{P_1}
\mathcal{P}_{T_1}
=\left[\begin{array}{ccc}
A_{1}     & 0 & 0 \\[1mm]
C_{1} &-S_2     & 0 \\[1mm]
 0             &C_{2}     &S_3
\end{array}\right] =\mathcal{L} \mathcal{D}.
\end{equation}

\begin{theorem}\label{Theo_tri}
Assume that all $S_i, i=1, 2, 3$, are invertible, then, for the preconditioner (\ref{P_1}),
\begin{equation}\label{U_part}
\mathcal{P}_{T_1}^{-1}\mathcal{A}= \mathcal{U},
\end{equation}
and $\mathcal{P}_{T_1}^{-1}\mathcal{A}$ satisfies the polynomial equation $(\lambda -1)^3=0$. If we further assume that $B_{1}^TS^{-1}_2 B_{2}^T$ is nonzero,
then $(\lambda -1)^3$ is the minimal polynomial.
\end{theorem}
\begin{proof}
From (\ref{LDUdecomp}), because $\mathcal{A}=\mathcal{L}\mathcal{D}\mathcal{U}$,
$\mathcal{P}_{T_1}^{-1}\mathcal{A}= (\mathcal{L} \mathcal{D})^{-1}\mathcal{A}= \mathcal{U}$, which is the block upper
triangular matrix of the form in (\ref{U_part}).

On one hand, $(\mathcal{P}_{T_1}^{-1}\mathcal{A}-I)^3=0$. To verify whether $(\lambda -1)^3$ is the minimal polynomial,
we assume that there exist nonzero numbers $a, b, c$ such that $a \mathcal{U}^2+b \mathcal{U}+cI=0$. We note that
$$
\mathcal{U}^2= \left[\begin{array}{ccc}
I     & 2A_{1}^{-1}B_{1}^T & -A_{1}^{-1}B_{1}^TS^{-1}_2 B_{2}^T \\[1mm]
0     & I    & -2S^{-1}_2 B_{2}^T\\[1mm]
 0    &0      & I
\end{array}\right],
$$
then
$$
a \mathcal{U}^2+b \mathcal{U}+c I =
 \left[\begin{array}{ccc}
(a+b+c))I     & (2a+b)A_{1}^{-1}B_{1}^T & -a A_{1}^{-1}B_{1}^TS^{-1}_2 B_{2}^T \\[1mm]
0     & (a+b+c)I    & -(2a+b)S^{-1}_2 B_{2}^T\\[1mm]
 0    &0      & (a+b+c)I
\end{array}\right]=0.
$$
Because both $A_1$ is invertible and $B_{1}^TS^{-1}_2 B_{2}^T$ is nonzero, from the $(1, 3)$ entry of $a \mathcal{U}^2+b \mathcal{U}+c I$, we derive that $a=0$. Then, from the fact that the $(1, 2)$ and $(2, 3)$ entries of $a \mathcal{U}^2+b \mathcal{U}+c I$ are zeros, we conclude $b=0$. Lastly,
from the diagonal entries, see that $c=0$. Therefore, $(\lambda -1)^3$ is the minimal polynomial of $\mathcal{P}_{T_1}^{-1}\mathcal{A}$.
\end{proof}

Other preconditioners which will make the corresponding preconditioned systems have eigenvalues $1$ or $-1$ are
\begin{equation}\label{P_234}
\mathcal{P}_{T_2}=\left[\begin{array}{ccc}
A_{1}     & 0 & 0 \\
C_{1} &S_2     & 0 \\
 0             &-C_{2}     &S_3 \\
\end{array}\right],
\quad
\mathcal{P}_{T_3}=\left[\begin{array}{ccc}
A_{1}     &0 & 0 \\
 C_{1} &S_2     &0  \\
 0             &-C_{2}     &-S_3 \\
\end{array}\right],
\quad
\mathcal{P}_{T_4}=\left[\begin{array}{ccc}
A_{2}     & 0 & 0 \\[1mm]
C_{1} &-S_2     & 0 \\[1mm]
 0             &C_{2}     &-S_3
\end{array}\right].
\end{equation}

{\bf Remark 1.} For preconditioners listed in (\ref{P_234}), we have the following conclusions.
\begin{equation*}
\mathcal{P}_{T_2}^{-1}\mathcal{A}=
\left[\begin{array}{ccc}
I     & A_{1}^{-1}B_{1}^T & 0 \\
0&-I     & S^{-1}_2 B_{2}^T \\
 0             &0     &I \\
\end{array}\right],
\end{equation*}
and $\mathcal{P}_{T_2}^{-1}\mathcal{A}$ satisfies the polynomial equation $(\lambda -1)^2(\lambda+1)=0$.
\begin{equation*}
\mathcal{P}_{T_3}^{-1}\mathcal{A}=
\left[\begin{array}{ccc}
I     & A_{1}^{-1}B_{1}^T & 0 \\
0 &-I     & S^{-1}_2 B_{2}^T \\
 0             &0    &-I \\
\end{array}\right],
\end{equation*}
and $\mathcal{P}_{T_3}^{-1}\mathcal{A}$ satisfies the polynomial equation $(\lambda -1)(\lambda +1)^2=0$.
\begin{equation*}
\mathcal{P}_{T_4}^{-1}\mathcal{A}=
\left[\begin{array}{ccc}
I     & A_{1}^{-1}B_{1}^T & 0 \\[1mm]
0 &I     & -S^{-1}_2 B_{2}^T\\[1mm]
 0             &0     &-I
\end{array}\right],
\end{equation*}
and $\mathcal{P}_{T_4}^{-1}\mathcal{A}$ satisfies the polynomial equation $(\lambda-1)^2(\lambda+1)=0$.


{\bf Definition}: A matrix is said to be stable if the real parts of all the eigenvalues are
negative. Correspondingly, if the real parts of all the eigenvalues of a matrix $A$ are positive, we say that these eigenvalues are positive real.
If a preconditioner $\mathcal{P}$ of $\mathcal{A}$ satisfies $\mathcal{P}^{-1}\mathcal{A}$ is positive real, we call $\mathcal{P}$ is a positive
stable preconditioner of $\mathcal{A}$.

A sufficient condition to guarantee the eigenvalues of a matrix $A$ is positive real is the Hermitian part of $A$, i.e., $\frac{1}{2}(A + A^H)$, is SPD. The reverse direction is not true. There are many matrices whose
eigenvalues are positive but their Hermitian parts are not SPD.

\begin{theorem} \label{Theo_diag}
Let the assumptions in Theorem \ref{Theo_tri} hold. We further assume that $A_{2}=0$ and $A_3=0$. If the block-diagonal preconditioner is
\begin{equation}\label{P_D1}
\mathcal{P}_{D_1}=\left[
\begin{array}{ccc}
A_{1} &0&0\\
0&S_2&0\\
0&0&S_3
\end{array}
\right],
\end{equation}
then,
\begin{equation}\label{preD_sys}
\mathcal{T}_1=\mathcal{P}_{D_1}^{-1}\mathcal{A}=
\left[\begin{array}{ccc}
I     &A_{1}^{-1} B^T_{1} & 0 \\[1mm]
S_2^{-1}C_{1} &0     & S_2^{-1}B^T_{2} \\[1mm]
 0             &S_3^{-1}C_{2}     &0
\end{array}\right],
\end{equation}
which satisfies the polynomial equation $(\lambda -1) (\lambda^2-\lambda-1)(\lambda^3-\lambda^2 -2\lambda +1)=0$.
\end{theorem}

\begin{proof}
By a simple calculation, one can derive that (\ref{preD_sys}) holds. Then, we see that
\begin{equation*}
\mathcal{T}_1^2=
\left[\begin{array}{ccc}
I+A_{1}^{-1} B^T_{1}S_2^{-1}C_{1}    & A_{1}^{-1} B^T_{1}           &A_{1}^{-1} B^T_{1}S_2^{-1}B^T_{2}\\[2mm]
S_2^{-1}C_{1}                                        &I+S_2^{-1}B^T_{2}S_3^{-1}C_{2}        &0\\[2mm]
S_3^{-1}C_{2}S_2^{-1}C_{1}                        &0                                              &I
\end{array}\right],
\end{equation*}

\begin{equation*}
\mathcal{T}_1^3=
\left[\begin{array}{ccc}
I+2A_{1}^{-1} B^T_{1}S_2^{-1}C_{1}  & A_{1}^{-1} B^T_{1}(2I+S_2^{-1}B^T_{2}S_3^{-1}C_{2})  &A_{1}^{-1} B^T_{1}S_2^{-1}B^T_{2}\\[2mm]
(2I+S_2^{-1}B^T_{2}S_3^{-1}C_{2})S_2^{-1}C_{1} & I & 2S_2^{-1}B^T_{2}\\[2mm]
S_3^{-1}C_{2}S_2^{-1}C_{1} &2S_3^{-1}C_{2}
& 0
\end{array}\right],
\end{equation*}
and
\begin{equation*}
\mathcal{T}_1-I=
\left[\begin{array}{ccc}
0                         &A_{1}^{-1} B^T_{1}        & 0 \\[2mm]
S_2^{-1}C_{1}    &-I                                         & S_2^{-1}B^T_{2} \\[2mm]
 0                        &S_3^{-1}C_{2}                          &-I \\
\end{array}\right],
\end{equation*}
\begin{equation*}
(\mathcal{T}_1^2-\mathcal{T}_1-I)=
\left[\begin{array}{ccc}
-I+A_{1}^{-1} B^T_{1}S_2^{-1}C_{1}    & 0           &A_{1}^{-1} B^T_{1}S_2^{-1}B^T_{2}\\[2mm]
0                                                              &S_2^{-1}B^T_{2}S_3^{-1}C_{2}        &- S_2^{-1}B^T_{2} \\[2mm]
S_3^{-1}C_{2}S_2^{-1}C_{1}                         &-S_3^{-1}C_{2}                                             &0
\end{array}\right],
\end{equation*}
\begin{equation*}
\mathcal{T}_1^3-\mathcal{T}_1^2-2\mathcal{T}_1+I=
\left[\begin{array}{ccc}
-I+A_{1}^{-1} B^T_{1}S_2^{-1}C_{1}       & A_{1}^{-1} B^T_{1}(-I+S_2^{-1}B^T_{2}S_3^{-1}C_{2})  &0\\[2mm]
(-I+S_2^{-1}B^T_{2}S_3^{-1}C_{2})S_2^{-1}C_{1}   & I-S_2^{-1}B^T_{2}S_3^{-1}C_{2}                                       & 0\\[2mm]
0 &0& 0
\end{array}\right].
\end{equation*}
Combining all of the above, one can verify that
\begin{equation*}\begin{array}{ll}
(\mathcal{T}_1^2-\mathcal{T}_1-I)(\mathcal{T}_1-I)=(\mathcal{T}_1-I)(\mathcal{T}_1^2-\mathcal{T}_1-I)=\\[4mm]
\qquad\qquad\qquad\left[\begin{array}{ccc}
0                                                          &A_{1}^{-1} B^T_{1}S_2^{-1}B^T_{2}S_3^{-1}C_{2}       & -A_{1}^{-1} B^T_{1}S_2^{-1}B^T_{2}\\[2mm]
S_2^{-1}B^T_{2}S_3^{-1}C_{2}S_2^{-1}C_{1} &-2S_2^{-1}B^T_{2}S_3^{-1}C_{2}                                       &2S_2^{-1}B^T_{2}\\[2mm]
-S_3^{-1}C_{2}S_2^{-1}C_{1}                    &2S_3^{-1}C_{2}                                                            & -I
\end{array}\right].
\end{array}
\end{equation*}
The $(i, j)$ block of $(\mathcal{T}_1-I)(\mathcal{T}_1^2-\mathcal{T}_1-I)(\mathcal{T}_1^3-\mathcal{T}_1^2-2\mathcal{T}_1+I)$ is denoted as $A_{ij}$. Because $A_2$ and $A_3$ are zeros, and
hence $S_i=C_{i-1} S_{i-1}^{-1} B_{i-1}^T, i=2, 3$. One can verify that
\begin{equation*}
\begin{array}{ll}
A_{11}=A_{1}^{-1} B^T_{1}S_2^{-1}B^T_{2}S_3^{-1}C_{2}(-I+S_2^{-1}B^T_{2}S_3^{-1}C_{2})S_2^{-1}C_{1}=0,\\[2mm]
A_{12}=A_{1}^{-1} B^T_{1}S_2^{-1}B^T_{2}S_3^{-1}C_{2}(I-S_2^{-1}B^T_{2}S_3^{-1}C_{2})=0,\\[2mm]
A_{21}=S_2^{-1}B^T_{2}S_3^{-1}C_{2}S_2^{-1}C_{1}(-I+A_{1}^{-1} B^T_{1}S_2^{-1}C_{1})\\[2mm]
      \qquad\quad-2S_2^{-1}B^T_{2}S_3^{-1}C_{2}(-I+S_2^{-1}B^T_{2}S_3^{-1}C_{2})S_2^{-1}C_{1}=0,\\[2mm]
A_{22}=S_2^{-1}B^T_{2}S_3^{-1}C_{2}S_2^{-1}C_{1}A_{1}^{-1} B^T_{1}(-I+S_2^{-1}B^T_{2}S_3^{-1}C_{2})\\[2mm]
       \qquad\quad-2S_2^{-1}B^T_{2}S_3^{-1}C_{2}(I-S_2^{-1}B^T_{2}S_3^{-1}C_{2})=0,\\[2mm]
A_{31}=-S_3^{-1}C_{2}S_2^{-1}C_{1}(-I+A_{1}^{-1} B^T_{1}S_2^{-1}C_{1})\\[2mm]
       \qquad\quad+2S_3^{-1}C_{2}(-I+S_2^{-1}B^T_{2}S_3^{-1}C_{2})S_2^{-1}C_{1}=0,\\[2mm]
A_{32}=-S_3^{-1}C_{2}S_2^{-1}C_{1}A_{1}^{-1} B^T_{1}(-I+S_2^{-1}B^T_{2}S_3^{-1}C_{2})\\[2mm]
\qquad\quad+2S_3^{-1}C_{2}(I-S_2^{-1}B^T_{2}S_3^{-1}C_{2})=0,\\[2mm]
A_{13}=0,\quad A_{23}=0, \quad A_{33}=0.
\end{array}
\end{equation*}
We therefore conclude that
\begin{equation}\label{pred1}
(\mathcal{T}_1-I)(\mathcal{T}_1^2-\mathcal{T}_1-I)(\mathcal{T}_1^3-\mathcal{T}_1^2-2\mathcal{T}_1+I)=0.
\end{equation}
\end{proof}

We note that (\ref{pred1}) can be factorized into distinct linear factors. Correspondingly, the roots of the above polynomial are:
$$
-1.2470,~ -0.6180, ~0.4450,~ 1, ~1.6180, ~1.8019.
$$
It follows that the eigenvalues of $\mathcal{T}_1$ must be among these values.
One can use the method of contradiction to prove that there is no polynomial of degree less than $6$ such that
$\mathcal{T}_{1}$ satisfies the polynomial equation.
Theorem 2.3 in \cite{Sogn2019Schur} shows that
if $\mathcal{A}$ is a symmetric matrix, then the eigenvalues of the preconditioned system $\mathcal{T}_1$ are
\begin{equation*}
\sigma_p(\mathcal{T}_1) = \left\lbrace  2\cos\left(\frac{2i-1}{2j+1}\pi\right) \colon j = 1,\ldots,3, \ i = 1,\ldots, j \right\rbrace,
\end{equation*}
Correspondingly, the condition number is
\begin{equation*}
\kappa(\mathcal{T}_1) \leq \frac{\cos\left(\frac{\pi}{7}\right)}{\sin\left(\frac{\pi}{14}\right)}\approx4.05 .
\end{equation*}
Our conclusion for eigenvalues in Theorem \ref{Theo_diag} is consistent with that in \cite{Sogn2019Schur}. However, the assumption that
$\mathcal{A}$ is symmetric and is removed in our analysis.

We also consider the following Schur complement based block-diagonal preconditioners:
\begin{equation}\label{P_D2to4}
\mathcal{P}_{D_2}=\left[
\begin{array}{ccc}
A_{1} &0   &0\\
0    &S_2  &0\\
0    &0    &-S_3
\end{array}
\right], \quad
\mathcal{P}_{D_3}=\left[
\begin{array}{ccc}
A_{1} &0 &0\\
0    &-S_2&0\\
0    &0&S_3
\end{array}
\right], \quad
\mathcal{P}_{D_4}=\left[
\begin{array}{ccc}
A_{1} &0&0\\
0&-S_2&0\\
0&0&-S_3
\end{array}
\right].
\end{equation}
Furthermore, if $A_{2}=0$ and $A_3=0$, we have the following conclusions: If the preconditioner is
$\mathcal{P}_{D_2}$, then
\begin{equation*}
\mathcal{T}_2=\mathcal{P}_{D_2}^{-1}\mathcal{A}=
\left[\begin{array}{ccc}
I     &A_{1}^{-1} B^T_{1} & 0 \\
S_2^{-1}C_{1} &0     & S_2^{-1}B^T_{2} \\
 0             &-S_3^{-1}C_{2}     &0 \\
\end{array}\right],
\end{equation*}
satisfies the polynomial equation $(\lambda-1)(\lambda^2-\lambda-1)(\lambda^3-\lambda^2-1)=0$;
Its roots are
$$
-0.618,~~ -0.2328 + 0.7926i, ~~ -0.2328 - 0.7926i,~~ 1, ~~1.4656,~~   1.618;
$$
If the preconditioner is $\mathcal{P}_{D_3}$, then
\begin{equation*}
\mathcal{T}_3=\mathcal{P}_{D_3}^{-1}\mathcal{A}=
\left[\begin{array}{ccc}
I     &A_{1}^{-1} B^T_{1} & 0 \\
-S_2^{-1}C_{1} &0     &- S_2^{-1}B^T_{2} \\
 0             &S_3^{-1}C_{2}     &0 \\
\end{array}\right],
\end{equation*}
which satisfies the polynomial equation $(\lambda-1)(\lambda^2-\lambda+1)(\lambda^3-\lambda^2+2\lambda-1)=0$;
Correspondingly, the roots are
$$
0.5698,~~ 1,~~   0.5 + 0.8660i,~~   0.5 - 0.8660i,~~   0.2151 + 1.3071i,~~   0.2151 - 1.3071i;
$$
If the preconditioner is $\mathcal{P}_{D_4}$, then
\begin{equation*}
\mathcal{T}_4=\mathcal{P}_{D_4}^{-1}\mathcal{A}=
\left[\begin{array}{ccc}
I     &A_{1}^{-1} B^T_{1} & 0 \\
-S_2^{-1}C_{1} &0     &- S_2^{-1}B^T_{2} \\
 0             &-S_3^{-1}C_{2}     &0 \\
\end{array}\right],
\end{equation*}
which satisfies the polynomial equation $(\lambda-1)(\lambda^2-\lambda+1)(\lambda^3-\lambda^2+1)=0$;
The roots are
$$
 -0.7549,~~ 1,~~   0.5 + 0.8660i,~~   0.5 - 0.8660i,~~   0.8774 + 0.7449i,~~   0.8774 - 0.7449i.
$$
We highlight here that all eigenvalues of $\mathcal{T}_3$ have positive real parts.
Therefore, $\mathcal{P}_{D_3}$ is a positively stable preconditioner for $\mathcal{A}$. We also comment here that when the $A_{1}$ and the Schur complements $S_2$ and $S_3$ are replaced by their inexact versions, some analysis on the eigenvalue distribution of the preconditioned systems are obtained in \cite{bradley2023eigenvalue}.

\section{Nested Schur complement based preconditioners for $n$-tuple block tridiagonal problems}\label{subsec:NestGSchur}
Now, we will discuss preconditioners for the $n$-tuple block tridiagonal problem of the following form.
  \begin{equation}
    \label{eq:ADefNN}
    \mathcal{A}_n =
    \left[\begin{array}{cccc}
      A_1 & B_1^T  &   & \\
      C_1 & -A_2  & \ddots &  \\
          &\ddots & \ddots & B_{n-1}^T \\[1ex]
          &       & C_{n-1}   & (-1)^{n-1}A_n
  \end{array}\right].
  \end{equation}
Similar to the $3$-by-$3$ block case, we define
$$ 
S_{i+1} = A_{i+1} + C_i S^{-1}_i B_i^T, \quad \text{for}\quad i = 1, 2,\ldots,n-1,
$$
with the initial setting $S_1 = A_1$. Here and hereafter, we assume that $S_i, i = 1, 2,\ldots,n$ are invertible.
The block tridiagonal system can be written recursively as
$$
\mathcal{A}_{i+1}=\left[\begin{array}{cccc}
   &\mathcal{A}_i\quad &\mathbb{B}_i^T\\
   &\mathbb{C}_i \quad &(-1)^{i}A_{i+1}
  \end{array}\right].
$$
Here,
$$
\mathbb{B}^T_i =
[0, . . . 0, B_i]^T,   \quad \quad \mathbb{C}_i = [0, . . . 0, C_i].
$$
Then, following the notations introduced in \cite{Sogn2019Schur},
$$
\begin{array}{ll}
   \mbox{Schur}(\mathcal{A}_{i+1})&=(-1)^{i}A_{i+1}-\mathbb{C}_i\mathcal{A}_i^{-1}\mathbb{B}_i^T\\[2mm]
                        &=(-1)^{i}(A_{i+1}+C_iS_i^{-1}B_i^T),
\end{array}$$
and $S_{i+1}=(-1)^{i}\mbox{Schur}(\mathcal{A}_{i+1})=A_{i+1}+C_iS_i^{-1}B_i^T$ with $S_1=A_1$.

Similar to the factorization (\ref{LDUdecomp}),
$\mathcal{A}_n$ can be factorized into
\begin{equation}\label{LDU_An}
\mathcal{A}_n =\mathcal{L}_n \mathcal{D}_n \mathcal{U}_n,
\end{equation}
where
\begin{equation}\label{Pre_Dn}
\mathcal{D}_n=\left[\begin{array}{cccc}
      A_1      & 0   &   & \\
      0      & -S_2   & \ddots &  \\
             &\ddots & \ddots & 0 \\[1ex]
             &       &0   & (-1)^{n-1}S_{n}
\end{array}\right],
\end{equation}
and
$$\begin{array}{ll}
\mathcal{L}_n=
    \left[\begin{array}{cccc}
      I      & 0   &   & \\
 C_1A_1^{-1} & I   & \ddots &  \\
             &\ddots & \ddots & 0 \\[1ex]
             &       & (-1)^{n-2}C_{n-1}S^{-1}_{n-1}   & I
  \end{array}\right]\quad
\mathcal{U}_n=\left[\begin{array}{cccc}
      I & A_1^{-1}B_1^T  &   & \\
      0 & I  & \ddots &  \\
          &\ddots & \ddots & (-1)^{n-2}S^{-1}_{n-1}B_{n-1}^T \\[1ex]
          &       & 0   & I
  \end{array}\right].
\end{array}
$$

\begin{theorem} Assume that all $S_{i},i=1,2,3,\ldots n$, are invertible. 
Let the preconditioner be
$$
\mathcal{P}_n =
  \left[\begin{array}{cccc}
    A_1 &  \\
    C_1 & -S_2 \\
     &  \ddots & \ddots \\
     & & C_{n-1} & (-1)^{n-1}S_{n}
  \end{array}\right]=\mathcal{L}_n\mathcal{D}_n,
$$
then
$
\mathcal{P}_n^{-1}\mathcal{A}_n=\mathcal{U}_n,$
and $\mathcal{P}_n^{-1}\mathcal{A}_n$ satisfies the polynomial $(\lambda-1)^n=0$.
Furthermore, if $\mathcal{A}_n$ is symmetric and $A_i,~i=1,2,\ldots,n$ are symmetric positive definite, let the preconditioner be $\mathcal{D}_n$, then all eigenvalues of $\mathcal{D}_n^{-1}\mathcal{A}_n$ are positive real.
\end{theorem}

\begin{proof}
Combining $\mathcal{L}_n\mathcal{D}_n$ in (\ref{LDU_An}), the first part of the conclusion is obvious.

For the second part, noting that if $\mathcal{A}_n$ is symmetric and $A_i,~i=1,2,\ldots,n$, are symmetric positive definite, then the matrix
\begin{equation}\label{Pre_Sn}
\mathcal{S}_n=\left[\begin{array}{cccc}
      S_1     &     &    &    \\
            &S_2  &    &    \\
             &   &\ddots  &  \\
               &       &    & S_{n}
\end{array}\right]
\end{equation}
is symmetric positive definite \cite{Sogn2019Schur}. We therefore can define $\mathcal{S}_n^{\frac{1}{2}}$, in which the $i$-th diagonal block is $S_i^{\frac{1}{2}}$.
To prove the eigenvalues of $\mathcal{D}_n^{-1}\mathcal{A}_n$ are positive real, we apply a similarity transform
to $\mathcal{D}_n^{-1}\mathcal{A}_n$, that is, left multiplying by $\mathcal{S}_n^{\frac{1}{2}}$ and right multiplying by $\mathcal{S}_n^{-\frac{1}{2}}$ to $\mathcal{D}_n^{-1}\mathcal{A}_n$ , we have
\begin{align}
\mathcal{S}_n^{\frac{1}{2}} \mathcal{D}_n^{-1} \mathcal{A}_n  \mathcal{S}_n^{-\frac{1}{2}}  &=
        \left[\begin{array}{cccc}
      I &    S_1^{-\frac{1}{2}}  B_1^T S_2^{-\frac{1}{2}}  &   & \\
      -S_2^{-\frac{1}{2}} B_1 S_1^{-\frac{1}{2}} &  S_2^{-\frac{1}{2}}A_2S_2^{-\frac{1}{2}}  & \ddots &  \\
          &\ddots & \ddots &  (-1)^{n-2} S_{n-1}^{-\frac{1}{2}}B_{n-1}^T S_n^{-\frac{1}{2}} \\[1ex]
          &       & (-1)^{n-1} S_n^{-\frac{1}{2}} B_{n-1} S_{n-1}^{-\frac{1}{2}} & S_n^{-\frac{1}{2}}A_n S_n^{-\frac{1}{2}}
  \end{array}\right]. \nonumber
\end{align}
We note that the Hermitian part of $\mathcal{S}_n^{\frac{1}{2}} \mathcal{D}_n^{-1} \mathcal{A}_n  \mathcal{S}_n^{-\frac{1}{2}}$ is
\begin{align}
\frac{1}{2} \left[ (\mathcal{S}_n^{\frac{1}{2}} \mathcal{D}_n^{-1} \mathcal{A}_n  \mathcal{S}_n^{-\frac{1}{2}})+(\mathcal{S}_n^{\frac{1}{2}} \mathcal{D}_n^{-1} \mathcal{A}_n  \mathcal{S}_n^{-\frac{1}{2}})^T \right]  &=
       \left[\begin{array}{cccc}
      I &     0 &   & \\
       0 &  S_2^{-\frac{1}{2}}A_2S_2^{-\frac{1}{2}}  & \ddots &  \\
         &\ddots & \ddots  & 0 \\[1ex]
         &       & 0  & S_n^{-\frac{1}{2}}A_n S_n^{-\frac{1}{2}}
  \end{array}\right] \nonumber
\end{align}
is symmetric positive definite. Therefore, we conclude that all eigenvalues of $\mathcal{D}_n^{-1}\mathcal{A}_n$ are positive real.

\end{proof}

{\bf Remark 2.} In \cite{Sogn2019Schur}, $\mathcal{S}_n$ is taken as a preconditioner of the block tridiagonal system $\mathcal{A}_n$.
The advantage is that $\mathcal{S}_n$ is SPD. However, the eigenvalues of $\mathcal{S}_n^{-1} \mathcal{A}_n$ may not be positive real. In comparison, $\mathcal{D}_n$ is a positively stable preconditioner for $\mathcal{A}_n$.

\begin{theorem}
Assume that $A_i=0$ for $i=2,3,\ldots, n$ in $\mathcal{A}_n$. If the Schur complement based preconditioner is
$\mathcal{D}_n$ in (\ref{Pre_Dn}), then the eigenvalues of $\mathcal{T}_n=\mathcal{D}_n^{-1}\mathcal{A}_n$ are the roots of polynomial
\begin{equation}\label{recurr_form}
\bar{p}_i=\lambda \bar{p}_{i-1}+\bar{p}_{i-2}, ~i=2,3,4,\ldots,n,
\end{equation}
with $\bar{p}_0(\lambda)=1,~\bar{p}_1(\lambda)=\lambda-1$. Moreover, all the roots of $\bar{p}_i,~i=1,2,\ldots,n$, lie in the right half-plane.
If $\mathcal{T}_n=\mathcal{D}_n^{-1}\mathcal{A}_n$ is further assumed to be diagonalizable, then $\mathcal{T}_n$ satisfies the polynomial equation
$\prod^{n}_{i=0}\bar{p}_i =0$.
\end{theorem}

\begin{proof}
To verify the first part of the theorem, we let
\begin{equation*}\begin{array}{ll}
\mathcal{C}_{i}=S^{-1}_{i+1}C_i, \quad \mathcal{B}^T_{i}=S^{-1}_iB_i^T, \quad\text{for } \, i = 1, 2,\ldots ,n-1.
\end{array}
\end{equation*}
Because $A_i=0$ for $i=2,\ldots,n$, we see that $\mathcal{C}_i\mathcal{B}^T_{i} = I$ and the preconditioned system $\mathcal{T}_n$ degenerates to
  \begin{equation}\label{preconded_sys}
    \mathcal{T}_n=\mathcal{D}_n^{-1}\mathcal{A}_n=
          \left[\begin{array}{cccc}

    I & \mathcal{B}^T_1  &  &\\
    -\mathcal{C}_1 & 0  & \ddots &\\
       & \ddots& \ddots& (-1)^{n-2}\mathcal{B}^T_{n-1}\\
        & & (-1)^{n-1}\mathcal{C}_{n-1}   & 0
\end{array}\right].
\end{equation}
Let $\lambda$ be an eigenvalue and ${\bf x}=(x_1, x_2, ..., x_n)^T$ be the corresponding eigenvector, we have $\mathcal{T}_n \mathbf{x} = \lambda \, \mathbf{x}$. More clearly,
\begin{align}
x_1 +\mathcal{B}^T_{1}x_2 &= \lambda x_1, \nonumber \\
\mathcal{C}_1x_1 +\mathcal{B}^T_{2}x_3 &= -\lambda x_2, \nonumber \\
&\vdots           \label{eigdetails}   \\
\mathcal{C}_{n-2}x_{n-2} +\mathcal{B}^T_{n-1}x_{n} &= (-1)^{n-2}\lambda x_{n-1}, \nonumber \\
\mathcal{C}_{n-1}x_{n-1} &= (-1)^{n-1}\lambda x_{n}. \nonumber
\end{align}
From the first equation of (\ref{eigdetails}),
\begin{equation*}
  \mathcal{B}^T_{1} x_2 = \bar{p}_1(\lambda) \, x_1 ,
  \quad \text{where} \quad \bar{p}_1(\lambda) = \lambda - 1.
\end{equation*}
Setting $x_2=x_3= \ldots = x_n= 0$, and $x_1\neq 0$, we note that $\lambda_{11} =1$, the root of $\bar{p}_1(\lambda)$, is an eigenvalue by
. If $\lambda \neq \lambda_{11}$, from the second equation and replacing $x_1$ by $\frac{1}{\bar{p}_1(\lambda)} \mathcal{B}^T_{1} x_2$, we then obtain
\[
  \mathcal{B}^T_{2}x_3
    = -\mathcal{C}_1 x_1 - \lambda \, x_2
    = -\frac{1}{\bar{p}_1(\lambda)} \,  \mathcal{C}_1 \mathcal{B}_1^T x_2 - \lambda \, x_2
    = -\frac{1}{\bar{p}_1(\lambda)} \,  x_2 - \lambda \, x_2
    = -R_2(\lambda) \, x_2,
\]
where
\[
  R_2(\lambda) = \lambda +\frac{1}{\bar{p}_1(\lambda)} = \frac{\bar{p}_2(\lambda)}{\bar{p}_1(\lambda)}
  \quad   \text{with} \quad \bar{p}_2(\lambda) = \lambda  \bar{p}_1(\lambda) +1=\lambda \bar{p}_1(\lambda) +\bar{p}_0(\lambda)  .
\]
By setting $x_3 = \ldots = x_n = 0$, we see that the two roots of the polynomial $\bar{p}_2(\lambda)$, denoted as
$\lambda_{21}$ and $\lambda_{22}$, are eigenvalues. Repeating this procedure leads to
\begin{align*}
  \mathcal{B}^T_{j}x_{j+1} = (-1)^{j-1}R_j(\lambda) \,  x_j,
  \text{ for } j=2,\ldots,n-1,
  \text{ and }
  0 = (-1)^{n-1}R_n(\lambda) \, x_n
  \text{ with } R_j(\lambda) = \frac{\bar{p}_j(\lambda)}{\bar{p}_{j-1}(\lambda)},
\end{align*}
where the polynomials $\bar{p}_j(\lambda)$ are recursively given by
$$
  \bar{p}_0(\lambda) = 1,\quad
  \bar{p}_1(\lambda) = \lambda - 1, \quad
  \bar{p}_{i+1}(\lambda) = \lambda \, \bar{p}_i(\lambda) +\bar{p}_{i-1}(\lambda)
  \quad \text{for} \, i \ge 1.
$$ 
Therefore, the eigenvalues of $\mathcal{T}_n$ are the roots of the polynomials $\bar{p}_1\left(\lambda \right),\bar{p}_2(\lambda),\ldots,\bar{p}_{n}\left(\lambda \right)$.

For proving the second part of the conclusion, we note that all eigenvalues of $\mathcal{D}^{-1}_n \mathcal{A}_n$ are the roots of $\bar{p}_i,~i=1,2,\ldots,n$.
We are going to show that all roots of the polynomials $\bar{p}_i, i=1,2,\ldots,n$, lie on the right half-plane by using the Routh-Hurwitz stability criterion \cite{barnett1977routh, gantmacher1959}.

The stability of a polynomial can be determined by the Routh-Hurwitz stability criterion which suggests using
a tabular method \cite{barnett1977routh}: For a general $k$th-degree polynomial
$$
p_k(\lambda)=a_k \lambda^k+a_{k-1} \lambda^{k-1}+\cdots+a_1 \lambda+a_0,
$$
its Routh table is a matrix of $k+1$ rows
, which has the following structure:
\begin{equation}\label{RH_stability}
R_{T_k}=
\begin{bmatrix}
   &r_{01}    & r_{02}   & r_{03}   & \cdots   \\[1mm]
 &r_{11}    & r_{12}   & r_{13}   & \cdots     \\[1mm]
  &r_{21}    & r_{22}   & r_{23}   & \cdots    \\[1mm]
  &\vdots    & \vdots   & \vdots   & \ddots
\end{bmatrix}
\end{equation}
In (\ref{RH_stability}), 
the first two rows, which come from the coefficients directly, are ordered as
\begin{equation}\label{routh_entries}
(r_{0j}) = (a_k, a_{k-2}, a_{k-4}, \cdots ), \quad (r_{1j}) = (a_{k-1}, a_{k-3}, a_{k-5}, \cdots ).
\end{equation}
Other entries are computed as follows:
$$
\displaystyle
r_{ij} = \frac{-1}{r_{i-1, 1}} \det
\begin{bmatrix}
r_{i-2, 1}    & r_{i-2, j+1} \\
r_{i-1, 1}    & r_{i-1, j+1}
\end{bmatrix}, \quad i=2, 3, 4, \cdots.
$$
When the table is completed, the number of sign changes in the first column is the number of roots in the right half-plane \cite{barnett1977routh}.
Note that the entries of the Routh table are uniquely determined
once the first two rows are fixed, i.e., the Routh table is uniquely dependent on the coefficients of $p_k$.

For the polynomial $\bar{p}_k$ generated by (\ref{recurr_form}), its key coefficients (for the leading terms and the constant) are
\begin{equation}\label{coef}
\bar{a}_k=1, \quad \bar{a}_{k-1}=-1,\quad \quad \bar{a}_0=(-1)^{k}.
\end{equation}
Such a conclusion can be proved by the method of mathematical induction by using the recurrence formula (\ref{recurr_form}). 
From (\ref{coef}), in general, we have.
$$
\bar{p}_k=\lambda^{k}-\lambda^{k-1}+\bar{a}_{k-2}\lambda^{k-2}+\bar{a}_{k-3}\lambda^{k-3}+\bar{a}_{k-4}\lambda^{k-4}\ldots+\bar{a}_{2}\lambda^{2}+\bar{a}_1\lambda+(-1)^{k}.
$$
For completing the Routh table of $\bar{p}_k$, let us write out the expression of $\bar{p}_{k-1}$ in terms of the coefficients of $\bar{p}_k$.
From the recurrence formula (\ref{recurr_form}), we note that the expression of $\bar{p}_{k-1}$ will have small differences for the case $k$ odd and the case $k$ even:
\begin{equation}\label{guilv}
\begin{array}{l}
\displaystyle \bar{p}_{k-1}=\lambda^{k-1}-\lambda^{k-2}-\bar{a}_{k-3}\lambda^{k-3}-(\bar{a}_{k-4}+\bar{a}_{k-5})\lambda^{k-4}-\bar{a}_{k-5}\lambda^{k-5}-(\bar{a}_{k-6}+\bar{a}_{k-7})\lambda^{k-6}\\
\displaystyle \qquad\quad\ldots-(\bar{a}_5+\bar{a}_4)\lambda^5-\bar{a}_4\lambda^4-(\bar{a}_3+\bar{a}_2)\lambda^3-\bar{a}_2\lambda^2-(\bar{a}_1+\bar{a}_0)\lambda-\bar{a}_0, \qquad\mbox{if $k$ is odd},\\[2mm]
\displaystyle \bar{p}_{k-1}=\lambda^{k-1}-\lambda^{k-2}-\bar{a}_{k-3}\lambda^{k-3}-(\bar{a}_{k-4}+\bar{a}_{k-5})\lambda^{k-4}-\bar{a}_{k-5}\lambda^{k-5}-(\bar{a}_{k-6}+\bar{a}_{k-7})\lambda^{k-6}\\
\displaystyle\qquad\quad\ldots -\bar{a}_5\lambda^5-(\bar{a}_4+\bar{a}_3)\lambda^4-\bar{a}_3\lambda^3-(\bar{a}_2+\bar{a}_1)\lambda^2-\bar{a}_1\lambda-1,\qquad\qquad\quad~~ \mbox{if $k$ is even}.
\end{array}
\end{equation}

For ease of presentation, let us assume that $k$ is even (then $k-1$ and $k+1$ are odd. The case $k$ is odd can be proved similarly by slightly changing the notations in the following arguments). We denote the matrix $R_{T_k}$ as
the Routh table of polynomial $\bar{p}_{k}$. The first two rows
of $R_{T_{k}}$ are
\begin{equation}\label{RTn_tworows}
\begin{bmatrix}
   & \bar{a}_k=1        & \bar{a}_{k-2}  &\bar{a}_{k-4} &\ldots   &\bar{a}_{2} & \bar{a}_0=1   \\
   & \bar{a}_{k-1}=-1   & \bar{a}_{k-3}  &\bar{a}_{k-5} &\ldots   &\bar{a}_1   &0             \\
\end{bmatrix};
\end{equation}
Then, the first two rows of $R_{T_{k-1}}$ are
\begin{equation}\label{RTnm1_tworows}
\begin{bmatrix}
       & 1    & -\bar{a}_{k-3}           &-\bar{a}_{k-5}          &\ldots   &-\bar{a}_{1}  \\
     & -1   &-(\bar{a}_{k-4}+\bar{a}_{k-5})  &-(\bar{a}_{k-6}+\bar{a}_{k-7}) &\ldots    &-1      \\
\end{bmatrix};
\end{equation}
the first two rows of $R_{T_{k+1}}$ are
\begin{equation}\label{RTn1_tworows}
\begin{bmatrix}
     & 1    & \bar{a}_{k-2}-\bar{a}_{k-1}  &\bar{a}_{k-4}-\bar{a}_{k-3}  &\ldots  &\bar{a}_{2}-\bar{a}_{3} &\bar{a}_{0}-\bar{a}_{1} \\
         &-\bar{a}_{k}=-1   & -\bar{a}_{k-2}         &-\bar{a}_{k-4}   &\ldots  &-\bar{a}_{2}        &-\bar{a}_0 \\
\end{bmatrix},
\end{equation}
and $R_{T_{k+1}}$ satisfy the following relationship.
\begin{equation}\label{routh_law}
R_{T_{k+1}}=
\begin{bmatrix}
      & 1    & \bar{a}_{k-2}-\bar{a}_{k-1} &\bar{a}_{k-4}-\bar{a}_{k-3}  &\ldots  &\bar{a}_{2}-\bar{a}_{3} &\bar{a}_{0}-\bar{a}_{1} \\
      \\
        &   &&-R_{T_{k}}  \\
       \\

\end{bmatrix}.
\end{equation}

One only needs to verify that the derivation from (\ref{RTn_tworows}) to (\ref{RTn1_tworows})
is correct. To check this, by using (\ref{recurr_form}), there holds
$$
\begin{array}{lll}
&\bar{p}_{k+1}=\lambda \bar{p}_{k}+ \bar{p}_{k-1}   \\
&= \lambda (\lambda^{k}-\lambda^{k-1}+\bar{a}_{k-2}\lambda^{k-2}+\bar{a}_{k-3}\lambda^{k-3}+\bar{a}_{k-4}\lambda^{k-4}\ldots+\bar{a}_{2}\lambda^{2}+\bar{a}_1\lambda+\bar{a}_0)  \\
&+\lambda^{k-1}-\lambda^{k-2}-\bar{a}_{k-3}\lambda^{k-3}-(\bar{a}_{k-4}+\bar{a}_{k-5})\lambda^{k-4}-\bar{a}_{k-5}\lambda^{k-5}-(\bar{a}_{k-6}+\bar{a}_{k-7})\lambda^{k-6}\ldots \\
& -(\bar{a}_5+\bar{a}_4)\lambda^5-\bar{a}_4\lambda^4-(\bar{a}_3+\bar{a}_2)\lambda^3-\bar{a}_2\lambda^2-(\bar{a}_1+\bar{a}_0)\lambda-\bar{a}_0  \\
&=\lambda^{k+1}-\lambda^{k}+(\bar{a}_{k-2}+1)\lambda^{k-1}+(\bar{a}_{k-3}-1)\lambda^{k-2}+(\bar{a}_{k-4}-\bar{a}_{k-3})\lambda^{k-3}-\bar{a}_{k-4}\lambda^{k-4} \\
&  \quad+(\bar{a}_{k-6}-\bar{a}_{k-5})\lambda^{k-5} -\bar{a}_{k-6}\lambda^{k-6}\ldots+(\bar{a}_4-\bar{a}_3)\lambda^2-\bar{a}_3\lambda^3+(\bar{a}_1-\bar{a}_2)\lambda^2-\bar{a}_1\lambda+1.
\end{array}
$$
As $k$ is even, $k+1$ is odd, the expression of $\bar{p}_{k+1}$ satisfies the first equation in (\ref{guilv}), and the first two rows of the Routh table of $\bar{p}_{k+1}$ are
$$
\begin{bmatrix}
     & 1    & \bar{a}_{k-2}-\bar{a}_{k-1}  &\bar{a}_{k-4}-\bar{a}_{k-3}  &\ldots  &\bar{a}_{2}-\bar{a}_{3}   &\bar{a}_{0}-\bar{a}_{1} \\
         &-\bar{a}_{k}=-1   & -\bar{a}_{k-2}         &-\bar{a}_{k-4}   &\ldots  &-\bar{a}_{2}        &-\bar{a}_0. \\
\end{bmatrix}.
$$
Thus, (\ref{RTn1_tworows}) is verified. As the entries of the Routh table are uniquely determined by the first two rows, the rest of the verifications are straightforward: for verifying (\ref{routh_law}), one can directly check the entries row-by-row by
using (\ref{routh_entries}); Actually, the entries of the $3$-rd row of $R_{T_{k+1}}$ are
$$
r_{21} =\frac{-1}{-1} \det
\begin{bmatrix}
1     & \bar{a}_{k-2} -\bar{a}_{k-1}  \\
-1    & -\bar{a}_{k-2}
\end{bmatrix} = -\bar{a}_{k-1}=1,   \quad
r_{22} =\frac{-1}{-1} \det
\begin{bmatrix}
1     & \bar{a}_{k-4} -\bar{a}_{k-3}  \\
-1    & -\bar{a}_{k-4}
\end{bmatrix} = -\bar{a}_{k-3},...,
$$
$$
r_{2j} =\frac{-1}{-1} \det
\begin{bmatrix}
1     & \bar{a}_{k-2j} -\bar{a}_{k-2j+1}  \\
-1    & -\bar{a}_{k-2j}
\end{bmatrix} = -\bar{a}_{k-2j+1}, \quad \mbox{for a general $j$}.
$$
By comparing these entries with the $2$nd row of (\ref{RTn_tworows}), one sees that (\ref{routh_law}) is verified.

We conclude that the entries of the first row of $R_{T_k}$ change values from $1$ to $-1$ alternatingly. Therefore,
the number of sign changes in the first column is $k$.
Thus, according to the Routh-Hurwitz stability criterion \cite{barnett1977routh, gantmacher1959} and the fact that there is no zero root,
the number of roots of $\bar{p}_k$ which are lying in the right half-plane equals $k$.
Therefore, all the roots of $\bar{p}_i,~i=1,2,\ldots,n$, lie in the right half-plane.

If $\mathcal{T}_n=\mathcal{D}_n^{-1}\mathcal{A}_n$ is further assumed to be diagonalizable, then there exists an invertible matrix, say $\mathcal{Q}_n$, such that
$\mathcal{Q}_n^{-1} \mathcal{T}_n \mathcal{Q}_n$ is a diagonal matrix with the diagonal entries being the eigenvalues. By a direct verification, $\mathcal{T}_n$ satisfies
the polynomial equation $\prod^{n}_{i=0}\bar{p}_i =0$.
\end{proof}

{\bf Remark 3.} We highlight here that in the assumption of Theorem 3.2, we do not require that $\mathcal{A}_n$ is symmetric.
Moreover, under the same assumption of Theorem 3.2, if $\mathcal{S}_n$ in (\ref{Pre_Sn}) is the preconditioner, then the eigenvalues of $\mathcal{S}_n^{-1} \mathcal{A}_n$ are the roots of the polynomial of the form
$\prod^{n}_{i=0}\tilde{p}_i$, where $\tilde{p}_0(\lambda)=1,~\tilde{p}_1(\lambda)=\lambda-1$ and
$$
\tilde{p}_i=\lambda \tilde{p}_{i-1}-\tilde{p}_{i-2}, ~i=2,3,4,\ldots,n.
$$
The proof of this conclusion is the same as that provided in \cite{Sogn2019Schur}, which is for the symmetric case.
However, $\mathcal{S}_n$ is not a positively stable preconditioner. In comparison, we show that $\mathcal{D}_n$ is a
positively stable preconditioner for $\mathcal{A}_n$ no matter whether it is symmetric or not.

{\bf Remark 4.} We comment here that if $\mathcal{A}_n$ is symmetric, the proof for verifying $\mathcal{D}_n$ is positively stable can
be carried out similarly to that for Theorem 3.1. More clearly, we left multiply by $\mathcal{S}_n^{\frac{1}{2}}$ and right multiply by $\mathcal{S}_n^{-\frac{1}{2}}$ to $\mathcal{D}_n^{-1}\mathcal{A}_n$, and denote
$\mathcal{H}_n =\mathcal{S}_n^{\frac{1}{2}} \mathcal{D}_n^{-1} \mathcal{A}_n  \mathcal{S}_n^{-\frac{1}{2}}$. Then,
we have
\begin{align}
\mathcal{H}_n &=
        \left[\begin{array}{cccc}
      I &    S_1^{-\frac{1}{2}}  B_1^T S_2^{-\frac{1}{2}}  &   & \\
      -S_2^{-\frac{1}{2}} B_1 S_1^{-\frac{1}{2}} &  0  & \ddots &  \\
          &\ddots & \ddots & (-1)^{n-2}  S_{n-1}^{-\frac{1}{2}}B_{n-1}^T S_n^{-\frac{1}{2}} \\[1ex]
          &       & (-1)^{n-1} S_n^{-\frac{1}{2}} B_{n-1} S_{n-1}^{-\frac{1}{2}} & 0
  \end{array}\right]. \nonumber
\end{align}
We note that the Hermitian part of $\mathcal{S}_n^{\frac{1}{2}} \mathcal{D}_n^{-1} \mathcal{A}_n  \mathcal{S}_n^{-\frac{1}{2}}$ is
\begin{align}
\frac{1}{2} \left[ (\mathcal{S}_n^{\frac{1}{2}} \mathcal{D}_n^{-1} \mathcal{A}_n  \mathcal{S}_n^{-\frac{1}{2}})+(\mathcal{S}_n^{\frac{1}{2}} \mathcal{D}_n^{-1} \mathcal{A}_n  \mathcal{S}_n^{-\frac{1}{2}})^T \right]  &=
       \left[\begin{array}{cccc}
      I &     0 &   & \\
       0 &  0  & \ddots &  \\
         &\ddots & \ddots  & 0 \\[1ex]
         &       & 0  & 0
  \end{array}\right]. \nonumber
\end{align}
Let us denote $\lambda$ as an eigenvalue of $\mathcal{H}_n$, $\bar{\lambda}$ as its conjugate,
 and ${\bf x}$ as the corresponding eigenvector. As $\mathcal{H}_n$ is invertible, $\lambda$ is nonzero and $\mathbf{x}$ is not a zero vector.
 We have
 $$
 \mathbf{x}^T \mathcal{H}_n \mathbf{x} = \lambda ||\mathbf{x}||^2,   \quad  \mathbf{x}^T \mathcal{H}_n^T \mathbf{x} = \bar{\lambda}||\mathbf{x}||^2.
 $$
Summing the two equations together and noting from the expression of $\mathcal{H}_n$, we have
\begin{align}
\mathbf{x}^T \left(\mathcal{H}_n+\mathcal{H}_n^T\right) \mathbf{x} & = 2 \mbox{Re} (\lambda) ||\mathbf{x}||^2,  \nonumber \\
2 ||x_1||^2  &= 2 \mbox{Re} (\lambda) ||\mathbf{x}||^2.   \label{realpart}
\end{align}
Here, $x_1$ is the first component of $\mathbf{x}$ and $\mbox{Re} (\lambda)$ is the real part of $\lambda$. We assert that $x_1$ is nonzero.
This conclusion can be proved by the method of contradiction. If $x_1$ is $0$, starting with the first equation of (\ref{eigdetails}), we have $x_2 \in \mbox{Ker}(\mathcal{B}_1^T)$.
However, $\mathcal{B}_1^T$ must be of full column rank because $\mathcal{B}_1 \mathcal{B}^T_1 = I$. It follows that $x_2$ is 0.
Repeating this procedure as presented in (\ref{eigdetails}), and noting that $\mathcal{B}_i^T$ is of full column rank as $\mathcal{B}_i \mathcal{B}^T_i = I$, we will have $x_3=...=x_n=0$. It contradicts the fact that $\mathbf{x}$ is nonzero. Therefore, from (\ref{realpart}), we conclude that $\mbox{Re} (\lambda)>0$.

\section{Additive Schur complement based preconditioners for twofold saddle point problems}\label{sec:AddSchur}

In this section, we discuss how to design preconditioners when the twofold saddle point system is permuted into (\ref{permuted_form1}).
Such a form of the linear system arises naturally from domain decomposition methods \cite{Toselli2006domain} with two subdomains.
Without confusion, we introduce simplified notation
\begin{equation}\label{permuted_sys}
\begin{array}{ll}
A=\left[\begin{array}{cc}
A_{1}    & 0 \\[1mm]
0   &A_3      \\[1mm]
\end{array}\right],\qquad
B^T=\left[\begin{array}{ccc}
 B^T_{1} \\[1mm]
 C_{2} \\[1mm]
\end{array}\right],\qquad
C=\left[\begin{array}{ccc}
C_{1}            &B^T_{2}
\end{array}\right].
\end{array}
\end{equation}
Here, $A_3$ can degenerate to $0$. If $A_1$ and $A_3$ are invertible, then $A$ is invertible, and
the Schur complement simply reads as
\begin{equation*}
S=A_{2}+CA^{-1}B^T.
\end{equation*}
Furthermore, if $S$ is assumed to be invertible, then the following results hold \cite{fischer1998minimum, ipsen2001note}.
\begin{proposition}
If $A$ and $S$ are invertible and the preconditioner is
$$
\mathcal{Q}_{1}=\left[\begin{array}{cc}
A    & 0 \\[1mm]
C   &-S      \\[1mm]
\end{array}\right],
$$
then
\begin{equation*}
\mathcal{Q}_{1}^{-1}\mathcal{A}=
\left[\begin{array}{cc}
I   & A^{-1}B^T \\[1mm]
0   &I      \\[1mm]
\end{array}\right],
\end{equation*}
which satisfies the polynomial equation $(\lambda -1)^2=0$.
Moreover, if $A_{2}= 0$ and the preconditioner is
$$\mathcal{Q}_{D_1}=\left[\begin{array}{cc}
A    & 0 \\[1mm]
0   &S      \\[1mm]
\end{array}\right],$$
then
$$
\mathcal{T}_{Q_{D_1}}=\mathcal{Q}_{D_1}^{-1}\mathcal{A}=
\left[\begin{array}{cc}
I    & A^{-1}B^T \\[1mm]
S^{-1}C   &0      \\[1mm]
\end{array}\right],
$$
which satisfies the polynomial equation $(\lambda(\lambda-1))^2=\lambda(\lambda-1)$.
The non-zero eigenvalues of $\mathcal{T}_{Q_{D_1}}$ are
$$
1,~~ \frac{1\pm\sqrt{5}}{2}.
$$
\end{proposition}

The proof of this proposition can be found in \cite{murphy2000note} and \cite{ipsen2001note}.

Based on the formulation (\ref{permuted_sys}), if the preconditioner is
$$
\mathcal{Q}_{2}=\left[\begin{array}{cc}
A   & 0 \\[1mm]
C   &S      \\[1mm]
\end{array}\right],
$$
then
\begin{equation*}
\mathcal{Q}_{2}^{-1}\mathcal{A}=
\left[\begin{array}{cc}
I   & A^{-1}B^T \\[1mm]
0   &-I      \\[1mm]
\end{array}\right],
\end{equation*}
which satisfies the polynomial equation $(\lambda -1)(\lambda +1) =0$.
Moreover, if $A_{2}= 0$ and the preconditioner is
$$\mathcal{Q}_{D_2}=\left[\begin{array}{cc}
A    & 0 \\[1mm]
0   &-S      \\[1mm]
\end{array}\right],$$
then
\begin{equation*}
\mathcal{T}_{Q_{D_2}}=\mathcal{Q}_{D_2}^{-1}\mathcal{A}=
\left[\begin{array}{cc}
I    & A^{-1}B^T \\[1mm]
-S^{-1}C   &0      \\[1mm]
\end{array}\right],
\end{equation*}
which satisfies the polynomial equation$(\lambda(\lambda-1))^2=-\lambda(\lambda-1)$.
The non-zero roots are
$$
1,~~ \frac{1\pm\sqrt{3}i}{2}.
$$

We note that the polynomial equations for $\mathcal{Q}_{i}$ and $\mathcal{Q}_{D_i}$ have a lower degree than $\mathcal{P}_{i}$ and $\mathcal{P}_{D_i}$ discussed in Section 2. From a purely algebraic point of view, if the GMRES method is applied, $\mathcal{Q}_{1}$ is more favored than $\mathcal{P}_{1}$, $\mathcal{Q}_{D_2}$ is more favored than $\mathcal{P}_{D_i}, i=1, 2, 3, 4$. However, on the other hand, it might be not easy to directly derive a proper approximation for the additive Schur complement for some specific problems \cite{mathew2008domain, Toselli2006domain}. For the nested Schur complement based preconditioners, they actually provide a recursive approach for solving block-tridiagonal systems.

\section{Additive Schur complement based preconditioners from the viewpoint of domain decomposition}
In this section, we generalize the discussion of Section 4 to the $n$-tuple case. We consider a system
   \begin{equation}
    \label{eq:A_N_ddm}
    \mathcal{A}_n =
    \left[\begin{array}{cccc}
      A_1 &0    &\ldots   &B_1^T \\
      0   &-A_2  &\ddots &\vdots  \\
    \vdots  &\ddots  &\ddots & B_{n-1}^T \\[1ex]
      C_1   &\ldots    & C_{n-1}   & (-1)^{n-1}A_n
  \end{array}\right].
  \end{equation}
Formally, the system matrix (\ref{eq:A_N_ddm}) arises naturally from the domain decomposition method for an elliptic
problem with $n-1$ subdomains. Like in Section 4, one can define $A$, $C$, $B^T$, and $D$.
The additive Schur complement preconditioner and the corresponding theory in Section 4 can still be applied. The resulting preconditioned system satisfies a polynomial of degree $2$ (block triangular preconditioner) or $4$ (block-diagonal preconditioner).

For some special forms of (\ref{eq:A_N_ddm}), for example, 3-by-3 block systems, it is possible to reorder \cite{Saad2003iterative, mathew2008domain} the
interior and interface unknowns subdomain by subdomain in a sequential manner and reduce them to the
block tridiagonal systems discussed in Section 3. If an additive Schur complement based preconditioner is applied,
then the degree of the polynomial equation that the preconditioned system satisfies is 2 (block triangular preconditioners)
or 4 (block-diagonal preconditioners). See \cite{murphy2000note, ipsen2001note} for the proof. In contrast, the preconditioners based on the nested Schur complement satisfy polynomials with degrees may be as high as $n$ (block triangular)
 or $n!$ (block-diagonal). Therefore, an additive Schur complement based preconditioner is more favored than a
nested Schur complement based preconditioner. The Schur complement of (\ref{eq:A_N_ddm}), which corresponds to the Steklov-Poincar\'{e} operator \cite{mathew2008domain, Toselli2006domain}, is an additive type. Algebraically inverting the Schur complement directly is typically challenging, as there is no reduction in operational costs. Instead, various preconditioning techniques such as Dirichlet-Neumann, (balancing) Neumann-Neumann, fractional Sobolev norm-based preconditioners, vertex space preconditioners, and others can be developed, particularly for elliptic problems. For more detailed information and discussions, interested readers are directed to \cite{mathew2008domain, Toselli2006domain}. For some specific PDE problems, for example, the Biot model \cite{lee2017parameter, oyarzua2016locking} or models from fluid
dynamics \cite{cai2009preconditioning, cai2015analysis, gatica2012twofold}, it is possible to derive an analytic approximation of
the additive type Schur complement from a Fourier analysis point of view \cite{cai2014efficient, cai2015analysis}. In this work, for a special 3-field reformulation of the Biot model, based on its block tridiagonal form, we design nested Schur complement based preconditioners. In our experiments, the Schur complements are firstly approximated by using a Fourier analysis derivation, then are further approximated by using incomplete Cholesky factorizations. The
detailed numerical experiments are provided in Section 6.

If $n>3$, we realize that not all block tridiagonal systems can be directly permuted into the domain decomposition
form (\ref{eq:A_N_ddm}), although the number of nonzero blocks of (\ref{eq:A_N_ddm}) and that of (\ref{eq:ADefNN})
are the same for the non-degenerate cases. Therefore, it is also important to investigate the nested Schur complement based preconditioners.

\section{Numerical experiments based on a 3-field formulation of the Biot model}
A 3-field formulation of the Biot model reads as \cite{lee2017parameter, oyarzua2016locking}:
	\begin{align}
		-2\mu \mbox{div}\left(\epsilonb(\bm{u})\right)+\nabla\xi=\bm{f},  \label{lee1}\\
		- \div \ub-\frac{1}{\lambda}\xi+\frac{\alpha}{\lambda} p=0,   \label{lee2}\\
		 \left(\left(c_0+\frac{\alpha^2}{\lambda}\right)p-\frac{\alpha}{\lambda}\xi\right)_t-K\div\left(\nabla p-\rho_f \bm{g}\right)=Q_s. \label{lee3}
	\end{align}
Here, $\bm{u}$ is the displacement, $p$ is the fluid pressure, and $\xi= -\alpha p + \lambda \mbox{div} \bm{u}$ is the total pressure with $\alpha$ being Biot constant, $K$ is
the permeability, $\bm{f}$ is the body force, $Q_s$ is the source term.
$$
\epsilonb({\bm{u}}):=\frac{1}{2}\left(\nabla \bm{u}+\nabla \bm{u}^T\right),
$$
and
\begin{equation}\label{Poisson_exp}
		\lambda = \frac{E \nu}
		{(1+\nu)(1-2\nu)},
		\quad  \mu=  \frac{E}
		{2(1+\nu)}.
	\end{equation}
The boundary conditions are
	\begin{align}
		\ub={\bm 0}     \quad  \mbox{on} ~\Gamma_d, \label{BC1}\\
		\sigma(\ub)\nb-\alpha p\nb=\hb \quad  \mbox{on} ~\Gamma_t, \label{BC2}\\
		p=0\quad  \mbox{on} ~\Gamma_p, \label{BC3} \\
		K\left(\nabla p-\rho_f\gb\right)\cdot\nb=g_2 \quad \mbox{on} ~\Gamma_f.   \label{BC4}
	\end{align}
The measures of the Dirichlet boundary $\Gamma_d$ and $\Gamma_p$ are assumed to be greater than $0$. For the wellposedness of the Biot problem with the above boundary
conditions, we refer the readers to \cite{lee2017parameter, oyarzua2016locking}.
Without loss of generality, the Dirichlet conditions in (\ref{BC1})-(\ref{BC4}) are assumed to be homogeneous. The initial conditions are:
	\begin{align} \label{ini}
		\bm u(0) = \bm u_{0} \quad \mbox{and} \quad p(0) = p_{0}.
	\end{align}

We apply the Taylor-Hood element pair for the discretization of $\bm{u}$ and $\xi$, and $P_1$ elements for $p$. The corresponding finite element spaces are denoted as $\bm{V}^h$, $Q^h$, and $W^h$. We apply a backward Euler scheme with a time stepsize $\Delta t$ for (\ref{lee3}). After the time and spatial discretizations, the resulting linear system is of the following form.
	\begin{equation}\label{couple_matrix}
		A x=
		\left[\begin{array}{ccc}
			A_{\bm{u}}     & B^{T}_{\bm{u} \xi} & 0 \\
			B_{\bm{u} \xi} &-A_{\xi}     & B^{T}_{\xi p} \\
			0             &B_{\xi p}     & A_p \\
		\end{array}\right]
		\left[\begin{array}{c}
			\bar{\bm{u}}    \\
			\bar{{\xi}} \\
			\bar{p}   \\
		\end{array}\right]
		=b.
	\end{equation}
Here, $\bar{\bm{u}}$, $\bar{{\xi}}$, and
			$\bar{p}$ are unknowns associated with $\bm{u}$, $\xi$, and $p$ respectively; $A_{\bm{u}}, B_{{\bm u}\xi},  A_{\xi}, B_{\xi p}$, and $A_{p}$ are the matrices resulted from the following
bilinear forms.
	\begin{align}
		a_{\bm{u}}(\bm{u}, \bm{v})=2\mu\left(\epsilonb(\ub),\epsilonb({\vb})\right), \quad b_{\bm{u}\xi}(\bm{u}, \phi)=-\left(\phi, {\rm div}\ub\right), \nonumber \\
		a_{\xi}(\xi, \phi)=\frac{1}{\lambda}(\xi, \phi),  \quad b_{\xi p}(\xi, \psi)= \frac{\alpha}{\lambda} \left(\xi, \psi\right), \nonumber \\
		a_{p}(p, \psi)=-\left( \left(c_0+\frac{\alpha^2}{\lambda}\right)p, \psi \right) - \Delta t K \left(\nabla p,\nabla\psi\right). \nonumber
	\end{align}
We comment here that $A_p$ is negative definite according to the weak form it corresponds to. Nevertheless, the presentation
in this section is consistent with the notations in the previous sections.

Applying a Fourier analysis derivation \cite{cai2014efficient, cai2015analysis, cai2022overlapping}: firstly, we note that the differential operator for solving (\ref{lee1})-(\ref{ini}) is
	$$
	\mathcal{A}=
	\left[\begin{array}{cccc}
		-2\mu\partial^2_x-\mu\partial^2_y   &-\mu\partial_y\partial_x           &\partial_x &0\\
		-\mu\partial_y\partial_x            &-\mu\partial^2_x-2\mu\partial^2_y  &\partial_y &0\\
		-\partial_x                         &-\partial_y                         &-1/\lambda  & \alpha/\lambda \\
		0                                   & 0                                 &\alpha/\lambda  &-(c_0+\frac{\alpha^2}{\lambda})+K \Delta t\Delta
	\end{array}\right].
	$$
The Fourier mode of $\hat{\mathcal{A}}$ is
	$$
	\hat{\mathcal{A}}  = \left[\begin{array}{cccc}
		2\mu k^2+\mu l^2    &\mu kl       &ik  &0\\
		\mu kl          & 2\mu l^2+\mu k^2 &il &0 \\
		-ik    & -il             & -\frac{1}{\lambda}  &\alpha/\lambda \\
		0  & 0 & \alpha/\lambda  &-(c_0+\frac{\alpha^2}{\lambda})-K \Delta t(k^2+l^2)
	\end{array}
	\right].
	$$
	Thus, the Fourier mode for $S_{\xi}=A_{\xi}+B_{\bm{u} \xi} A_{\bm{u}}^{-1} B_{\bm{u} \xi}^T$ is
	\begin{equation*}
		\begin{array}{rl}
			\displaystyle
			\hat{\mathbb{S}}_{\xi}&\displaystyle=\frac1\lambda+ [-ik ~ -il] \left[\begin{array}{cc}
				2\mu k^2+\mu l^2    &\mu kl  \\
				\mu kl          & 2\mu l^2+\mu k^2 \\
			\end{array}
			\right]^{-1}
			\left[\begin{array}{c}
				ik     \\
				il   \\
			\end{array}
			\right] 
			=\frac{1}{\lambda}+\frac{1}{2\mu}.
		\end{array}
	\end{equation*}
	Then, the Fourier mode for $S_{p}=A_p + B_{\xi p} S_{\xi}^{-1} B_{\xi p}^{T} $ is
	\begin{equation*}
		\displaystyle
		\begin{array}{rl}
			\hat{\mathbb{S}}_p &= -\left(c_0+\frac{\alpha^2}{\lambda} \right)- K \Delta t(k^2+l^2)+\frac{\alpha}{\lambda} \left(\frac{1}{\lambda}+\frac{1}{2 \mu} \right)^{-1}\frac{\alpha}{\lambda} \\
			&=  -(c_0+\frac{\alpha^2}{\lambda})-K\Delta t(k^2+l^2) + \frac{2\mu\alpha^2}{\lambda(\lambda + 2\mu)} \\
			&=  -(c_0+\frac{\alpha^2}{2\mu+\lambda})-K\Delta t(k^2+l^2).
		\end{array}
	\end{equation*}
	From the above Fourier analysis, it is easy to see that $\hat{\mathbb{S}}_{\xi}$ and $\hat{\mathbb{S}}_p$ correspond to the following matrices.
$$
\mathbb{S}_{\xi} =\left(\frac{1}{\lambda}+\frac{1}{2\mu} \right) M_{\xi} ~\quad \mbox{and} \quad~ \mathbb{S}_p=A_p+\frac{2\mu\alpha^2}{\lambda(\lambda + 2\mu)}M_p.
$$
Here, $M_{\xi}$ and $M_p$ are mass matrices that correspond to the identity operators in $Q^h$ and $W^h$, respectively.
Therefore, one can take the above two as the approximations of $S_{\xi}$ and $S_{p}$, respectively. In the implementation, one only needs to generate the
mass matrices in $Q^h$ and $W^h$ and put the proper scalings as shown above. In alignment with the formulations for $\mathcal{P}_{T_1}$, $\mathcal{P}_{T_2}$, $\mathcal{P}_{T_3}$, and $\mathcal{P}_{T_4}$ presented in equations (2.1) and (2.3), the corresponding preconditioners are represented as $\bar{\mathcal{P}}_{T_1}$, $\bar{\mathcal{P}}_{T_2}$, $\bar{\mathcal{P}}_{T_3}$, and $\bar{\mathcal{P}}_{T_4}$, respectively. This representation is achieved by substituting the exact Schur complements with the aforementioned approximations.
Likewise, for the counterparts $\mathcal{P}_{D_1}$, $\mathcal{P}_{D_2}$, $\mathcal{P}_{D_3}$, and $\mathcal{P}_{D_4}$ in equations (2.4) and (2.7), the resulting preconditioners are denoted as $\bar{\mathcal{P}}_{D_1}$, $\bar{\mathcal{P}}_{D_2}$, $\bar{\mathcal{P}}_{D_3}$, and $\bar{\mathcal{P}}_{D_4}$, respectively.

	\begin{table}[!htbp]
		\caption{Comparisons of the different preconditioners.
			GMRES iteration counts (and CPU time in seconds). The drop tolerance for the incomplete Cholesky factorizations is set to be $1.0 E -3$.}
		\centering\scriptsize
		\begin{tabular}{ccccc||cccccccc}
			\hline
			&\multicolumn{4}{c||}{Block Diagonal Precconditioners} &\multicolumn{4}{c}{Block Triangular Precconditioners}\\ \hline
			\multirow{1}*{$h$}
			& \multicolumn{1}{c}{$\widehat{\mathcal{P}}_{D_1}^{-1}$}&\multicolumn{1}{c}{~$\widehat{\mathcal{P}}_{D_2}^{-1}$}
			 &\multicolumn{1}{c}{\color{blue}$\widehat{\mathcal{P}}_{D_3}^{-1}$}&\multicolumn{1}{c||}{~$\widehat{\mathcal{P}}_{D_4}^{-1}$}
			 &\multicolumn{1}{c}{\color{blue}$\widehat{\mathcal{P}}_{T_1}^{-1}$}&\multicolumn{1}{c}{~$\widehat{\mathcal{P}}_{T_2}^{-1}$}
			 &\multicolumn{1}{c}{$\widehat{\mathcal{P}}_{T_3}^{-1}$}&\multicolumn{1}{c}{~$\widehat{\mathcal{P}}_{T_4}^{-1}$}\\
    		\hline
			$\frac{1}{16}$   &55 (0.0135)  &56 (0.0129)   &{\color{blue}50 (0.0122)}  &67 (0.0172)       &{\color{blue} 27 (0.0069)} &49 (0.0118)  &49 (0.0116)  &36 (0.0092)  \\
			$\frac{1}{32}$   &108 (0.2671)  &108 (0.2605) &{\color{blue}96 (0.2213)}  &127 (0.3559)      &{\color{blue}55 (0.1004)}  &97 (0.2237)  &97 (0.2187) &72 (0.1350)\\
			$\frac{1}{64}$   &229 (2.645)  &229 (2.586)  &{\color{blue}204 (2.170)} &273 (3.538)     &{\color{blue}118 (0.936)} &206 (2.225)  &206 (2.196) &157 (1.432)\\
			$\frac{1}{128}$  &470 (28.94)  &470 (29.03)  &{\color{blue}449 (28.03)} &593 (62.61)       &{\color{blue}261 (11.00)} &434 (42.44) &434 (42.83) &345 (28.44)\\
			$\frac{1}{256} $ &1057 (1027)  &1067 (1029) &{\color{blue}934 (957.9)} &1767 (1765) &{\color{blue}569 (382.1)} &947 (980.6) &947 (977.5) &782 (685.1) \\		
        \hline
	\end{tabular}
        \label{table1}
	\end{table}

	\begin{table}[!htbp]
		\caption{Comparisons of the different preconditioners.
			GMRES iteration counts (and CPU time in seconds). The drop tolerance for the incomplete Cholesky factorizations is set to be $1.0 E -4$.}
        \centering\scriptsize
		\begin{tabular}{ccccc||cccccccc}
			\hline
			&\multicolumn{4}{c||}{Block Diagonal Precconditioners}&\multicolumn{4}{c}{Block Triangular Precconditioners}\\ \hline
			\multirow{1}*{$h$}
			& \multicolumn{1}{c}{$\widehat{\mathcal{P}}_{D_1}^{-1}$}&\multicolumn{1}{c}{~$\widehat{\mathcal{P}}_{D_2}^{-1}$}
			 &\multicolumn{1}{c}{\color{blue}$\widehat{\mathcal{P}}_{D_3}^{-1}$}&\multicolumn{1}{c||}{~$\widehat{\mathcal{P}}_{D_4}^{-1}$}
			 &\multicolumn{1}{c}{\color{blue}$\widehat{\mathcal{P}}_{T_1}^{-1}$}&\multicolumn{1}{c}{~$\widehat{\mathcal{P}}_{T_2}^{-1}$}
			 &\multicolumn{1}{c}{$\widehat{\mathcal{P}}_{T_3}^{-1}$}&\multicolumn{1}{c}{~$\widehat{\mathcal{P}}_{T_4}^{-1}$}\\
			\hline
		$\frac{1}{16} $ &36 (0.0127)  &36 (0.0114)  &{\color{blue}32 (0.0118)}  &41 (0.0139)        &{\color{blue}15 (0.0073)}  &31 (0.0111)   &31 (0.0110)    &19 (0.0085) \\
			$\frac{1}{32} $  &51 (0.3073)  &51 (0.2624)  &{\color{blue}45 (0.2210)} &57 (0.3436)         &{\color{blue}24 (0.0961)}  &45 (0.2188) &45 (0.2415)  &29 (0.1259) \\
			$\frac{1}{64}$   &92 (0.988)  &92 (0.947)  &{\color{blue}80 (0.810)} &97 (1.007)         &{\color{blue}45 (0.414)}  &84 (0.854) &84 (0.841)  &54 (0.502) \\
			$\frac{1}{128}$ &187 (8.88) &188 (8.96)  &{\color{blue}167 (7.83)} &205 (10.54)    &{\color{blue}97 (4.31)}  &170 (7.81) &170 (7.88)  &119 (5.05) \\
			$\frac{1}{256}$   &373 (89.18) &373 (86.59)  &{\color{blue}353 (82.71)} &455  (119.4) &{\color{blue}206 (39.28)} &350 (80.22)  &350 (79.07)  &267 (54.99) \\
        \hline
		\end{tabular}
     \label{table2}
	\end{table}

For practical implementation, for each of $A_{\bm{u}}$, $\mathbb{S}_{\xi}$ and $-\mathbb{S}_p$ (because $\mathbb{S}_p$ is negative definite), we apply an incomplete Cholesky factorization, i.e.,
$$
\widehat{A}= L L^T + R
$$
with a drop tolerance. Then $\widehat{A}^{-1}$ is approximated by $L^{-T} L^{-1}$. We comment here that other types of inexact solvers can also be applied, for example, the methods introduced in \cite{cao2021cell, cao2022additive, neytcheva2011element}.
We denote the corresponding inexact approximations of $\bar{\mathcal{P}}_{D_1}$ to $\bar{\mathcal{P}}_{D_4}$ as $\widehat{\mathcal{P}}_{D_1}$ to $\widehat{\mathcal{P}}_{D_4}$. Similarly, we denote the corresponding inexact approximations of $\bar{\mathcal{P}}_{T_1}$ to $\bar{\mathcal{P}}_{T_4}$ as $\widehat{\mathcal{P}}_{T_1}$ to $\widehat{\mathcal{P}}_{T_4}$.
The implementation code, numerical results, and eigenvalues calculations are available at \url{https://github.com/cmchao2005/Preconditioners-for-Biot-model}.

In our numerical experiments, $\nu=0.499$, and $\lambda$ and $\mu$ are computed by using (\ref{Poisson_exp}). All the other physical parameters are set to be $1$. The computational domain is a unit square. The Dirichlet type boundaries $\Gamma_d$ and $\Gamma_p$ are the two vertical lines. Other parts of the boundary are Neumann-type.
In Table \ref{table1}, we report the GMRES iteration counts for all preconditioners with the drop tolerance for the incomplete factorizations being $10^{-3}$.
In Table \ref{table2}, we summarize the GMRES iteration counts for all preconditioners with the drop tolerance for the incomplete factorizations being $10^{-4}$. In all tests, the stopping criterion for the GMRES method is set to be that the $l^2$ norm of the relative residual is smaller than $1.0E-6$. From the numerical experiments, we see that the results based on $\widehat{\mathcal{P}}_{D_1}^{-1}$ are almost the same as those based on $\widehat{\mathcal{P}}_{D_2}^{-1}$,
the results based on $\widehat{\mathcal{P}}_{T_2}^{-1}$ are almost the same as those based on $\widehat{\mathcal{P}}_{T_3}^{-1}$. In comparison, $\widehat{\mathcal{P}}_{D_3}$ outperforms among
the block diagonal preconditioners, and $\widehat{\mathcal{P}}_{T_1}$ outperforms among the block triangular preconditioners. Note that we have theoretically proved that $\mathcal{P}_{D_3}$ and $\mathcal{P}_{T_1}$
are positively stable preconditioners. The numerical experiments clearly illustrate the advantages of the positively stable preconditioners.
Although we do not rigorously prove that $\bar{\mathcal{P}}_{D_3}$ and $\bar{\mathcal{P}}_{T_1}$ are positively stable preconditioners in this draft,
the corresponding numerical spectral verifications are provided in our code.

\section{Concluding remarks}\label{sec:Conclusion}

In this paper, both nested Schur complement and additive Schur complement based preconditioners are constructed for the twofold and block tridiagonal linear systems. The polynomial equations of the preconditioned matrices are analyzed. It is shown that by properly selecting the sign in front of each Schur complement, some preconditioners are positively stable. Numerical experiments based on the Biot model are provided to show that positively stable preconditioners outperform other preconditioners. More clearly, when inexact elliptic approximations are incorporated, the inexact versions of the positively stable preconditioners are more efficient.

For block tridiagonal systems, by comparing the theoretical analysis for the nested Schur complement based preconditioners and that for the additive type preconditioners, our argument is that permutation is important and necessary when designing preconditioners.
These results are instructive for devising the corresponding inexact versions of the preconditioners and iterative methods \cite{beik2018iterative}.

\section*{Acknowledgments} The authors would like to thank Mingjian Ding, and Baoxuan Zhu for providing an alternative proof of the Hurwitz stability of polynomials (\ref{recurr_form}). They also thank Jarle Sogn for communicating on Schur complement based preconditioners.
The work of M. Cai is partially supported by the NIH-RCMI grant through 347 U54MD013376, the affiliated project award from the Center for Equitable Artificial Intelligence and Machine Learning Systems (CEAMLS) at Morgan State University (project ID 02232301), and the National Science Foundation awards (1831950). The work of G. Ju is supported in part by the National Key R \& D Program of China (2017YFB1001604). The work of J. Li is partially supported by the National Natural Science Foundation of China No. 11971221 and the Shenzhen Sci-Tech Fund No. RCJC20200714114556020, JCYJ20170818153840322 and JCYJ20190809150413261, and Guangdong Provincial Key Laboratory of Computational Science and Material Design No. 2019B030301001.

\section*{Conflict of interest statement}
The authors declare that they have no conflict of interest.

\begin{appendices}
\section*{An alternative analysis of the Hurwitz stability of polynomials (\ref{recurr_form})}
\setcounter{section}{0}
\setcounter{equation}{0}
\setcounter{theorem}{0}

\renewcommand{\theequation}{A.\arabic{equation}}
\renewcommand{\thesection}{A.\arabic{section}}
\renewcommand{\thetheorem}{A.\arabic{theorem}}
\renewcommand{\thelemma}{A.\arabic{lemma}}

A multivariate  $f \in \mathbb{C}[z_1, \ldots, z_n]$ is said to be {\it weakly Hurwitz stable} if $f$ is either identically zero
or nonvanishing whenever $\Re(z_i) > 0$ for any $i \in [n]$, where
$\Re(z)$ is the real part of $z$ for $z \in \mathbb{C}$.
Denote by $\left\{r_i\right\}$ and $\left\{s_j\right\}$ the all real zeros
of $f, g \in \mathbb{R}\left[x \right]$, respectively.
We say that $g$ \emph{alternates left of} $f$ if $\deg(f)=\deg(g)=n$ and
\begin{equation}\label{1}
	s_{n} \leq r_{n} \leq \cdots \leq s_{2} \leq r_{2} \leq s_{1} \leq r_{1},
\end{equation}
and that $g$ \emph{interlaces} $f$ if $\deg(f)=\deg(g)+1=n$ and
\begin{equation}\label{2}
	r_{n} \leq s_{n-1} \leq \cdots \leq s_{2} \leq r_{2} \leq s_{1} \leq r_{1}.
\end{equation}
Let $g \preceq f$ denote either $g$ alternates left $f$ or $g$ interlaces
of $f$. Here, we denote $g \ll f$ if $g \preceq f$ (resp., $f \preceq g$)
and the leading coefficients of $f, g$ have the same (resp., opposite) sign.

Assume that a polynomial $f(x)=\sum_{k=0}^nf_kx^k$ has degree $n$,
then let
$$
f^E(x)=\sum_{k=0}^{\lfloor n/2 \rfloor}f_{2k}x^k
\quad \text{and} \quad
f^O(x)=\sum_{k=0}^{\lfloor (n-1)/2 \rfloor}f_{2k+1}x^k.
$$
\begin{theorem}\emph{(Hermite-Biehler Theorem)}\cite[Theorem 6.3.4]{RS02}\label{thm_HS}
A polynomial $f(x)$ with real coefficients is weakly Hurwitz stable if and only if $f^E(x)$ and $f^O(x)$
have only real and nonpositive zeros, and $f^O(x) \ll f^E(x)$.
\end{theorem}


\begin{theorem}
Let $p_n(x)$ satisfy the recurrence relation (\ref{recurr_form}) with
$p_0(x)=1$, $p_1(x)=x-1$.
Then, for any $n \geq 0$, all zeros of polynomial $p_n(x)$ lie in the right half-plane.
\end{theorem}

\begin{proof}
Obviously, all zeros of $p_n(x)$ lie in the right half-plane if and only if all zeros of $p_n(-x)$ lie in the left closed half-plane, i.e., weakly Hurwitz stable. Let
$$
f_n(x)=\sum_{k=0}^nf_kx^k:=(-1)^np_n(-x).
$$
Hence, $f_n(x)$ satisfies recurrence relation
\begin{equation}\label{rec+poly+f}
f_n(x)=xf_{n-1}(x)+f_{n-2}(x),
\end{equation}
where $f_0(x)=1$ and $f_1(x)=x+1$.
Note that $f_n(x)$ has only nonnegative coefficients for all $n \geq 0$.

\begin{table}[ht]\caption{The terms of $p_n(x)$ for $1 \leq n \leq 4$.}\label{table+poly+n}
\begin{center}
     \begin{tabular}{ccc}
 	\toprule[1.5pt]
 	$n$ & $p_n(x)$ \\
 	\midrule[0.5pt]
 	1 & $x-1$ \\
 	2 & $x^2-x+1$ \\
 	3 & $x^3-x^2+2x-1$ \\
 	4 & $x^4-x^3+3x^2-2x+1$\\
 	\bottomrule[1.5pt]
 \end{tabular}
 \end{center}
\end{table}

Let
$$
f_n^{E}(x)=\sum_{k=0}^{\lfloor n/2 \rfloor}f_{2k}x^k
\quad \text{and} \quad
f_n^{O}(x)=\sum_{k=0}^{\lfloor(n-1)/2\rfloor}f_{2k+1}x^k.
$$
Due to Theorem \ref{thm_HS}, it suffices to show that $f_n^{O}(x) \ll f_{n}^{E}(x)$ for any $n \geq 0$.
By the recurrence \eqref{rec+poly+f}, we get that
\begin{align}
	f_{n}^{E}(x) &= f_{n-2}^{E}(x)+xf_{n-1}^{O}(x), \label{rec+poly+f+even}\\
	f_{n}^{O}(x) &= f_{n-1}^{E}(x)+f_{n-2}^{O}(x). \label{rec+poly+f+odd}
\end{align}
By induction on $n$, it is easy to see that
\begin{equation}\label{rel+poly+f+parity}
f_{2n}^{E}(x)=f_{2n+1}^{E}(x) \quad \text{and} \quad f_{2n-1}^{O}(x)=f_{2n}^{O}(x).
\end{equation}

Subsequently, we will prove the following results through an induction argument based on $n$.
\begin{eqnarray*}
&&(1)\,f_{n-2}^{E}(x) \ll f_{n}^{E}(x),\quad \quad
(2)\,f_{n-2}^{O}(x) \ll f_{n}^{O}(x), \\
&&(3)\,f_{n-1}^{O}(x) \ll f_{n-2}^{E}(x),\quad \;
(4)\,f_{n-1}^{O}(x) \ll f_{n}^{E}(x).
\end{eqnarray*}
Firstly, $f_n^{O}(x) \ll f_n^{E}(x)$ follows from interlacing relations (3) and (4) by \eqref{rel+poly+f+parity}.
When $n=4$, by Table \eqref{table+poly+n} it is routine to verify that (1) -- (4) are true by Table \eqref{tab+f+parity}.
\begin{table}[ht]\caption{The terms of $f_n^{O}(x)$ and $f_n^{E}(x)$ for $2 \leq n \leq 4$.}\label{tab+f+parity}
\begin{center}
    \begin{tabular}{ccc}
		\toprule[1.5pt]
		$n$ & $f_n^{O}(x)$& $f_{n}^{E}(x)$ \\
		\midrule[0.5pt]
		2 & 1 & $x+1$ \\
		3 & $x+2$ & $x+1$ \\
		4 & $x+2$ & $x^2+3x+1$\\
		\bottomrule[1.5pt]
	\end{tabular}
 \end{center}
\end{table}
Assume that (1) -- (4) are true for an $n > 4$. We need to prove that
\begin{eqnarray*}
	&&(a)\,f_{n-1}^{E}(x) \ll f_{n+1}^{E}(x), \quad
	(b)\,f_{n-1}^{O}(x) \ll f_{n+1}^{O}(x), \\
	&&(c)\,f_{n}^{O}(x) \ll f_{n-1}^{E}(x),\quad \;\;\;
	(d)\,f_{n}^{O}(x) \ll f_{n+1}^{E}(x).
\end{eqnarray*}
By \eqref{rec+poly+f+odd} and the interlacing relations
$$
f_n^{O}(x) \ll  f_{n-1}^{E}(x)+f_{n-2}^{O}(x)
\quad \text{and} \quad
f_{n}^{O}(x) \ll -f_{n-2}^{O}(x),
$$
we have $f_{n}^{O}(x) \ll f_{n-1}^{E}(x)$, i.e., (c) holds.
Moreover, by \eqref{rec+poly+f+even} and combining $$f_{n-1}^E(x) \ll xf_{n}^O(x)
\quad \text{and} \quad
f_{n-1}^E(x) \ll f_{n-1}^E(x)
$$
we derive that $f_{n-1}^E(x) \ll f_{n+1}^E(x)$, i.e., (a) holds.
Note that $f_{n+1}^{O}(x) = f_{n}^{E}(x)+f_{n-1}^{O}(x)$,
hence (b) follows from (4) and $f_{n-1}^{O}(x) \ll f_{n-1}^{O}(x)$.
Similarly, due to $f_{n+1}^E(x) = f_{n-1}^E(x) + xf_{n}^O(x)$ by \eqref{rec+poly+f+even}, (d) is immediately derived by (c) and $f_{n}^O(x) \ll xf_n^O(x)$.
This completes the proof.
\end{proof}

\end{appendices}

%

\medskip
\medskip

\penalty-8000

\bibliographystyle{siam}


\end{document}